\documentclass[11pt]{amsart}
\usepackage{graphicx}

\newtheorem{thm}{Theorem}[section]
\newtheorem{lemma}[thm]{Lemma}

\newcommand{\R}{\mathbb{R}}
\newcommand{\C}{\mathbb{C}}

\newcommand{\N}{\mathbb{N}}

\newcommand{\E}[2]{E^{(#1)}_{#2,\n,\p}}
\newcommand{\EE}[2]{{\mathcal E}^{(#1)}_{#2,\n,\p}}
\newcommand{\F}[1]{F^{(#1)}_{#1,\n,\p}}
\newcommand{\FF}[1]{{\mathcal F}^{(#1)}_{#1,\n,\p}}
\renewcommand{\P}[1]{P^{(#1)}_{#1,\n,\p}}
\newcommand{\PP}[1]{{\mathcal P}^{(#1)}_{#1,\n,\p}}
\newcommand{\Q}[1]{Q^{(#1)}_{#1,\n,\p}}
\newcommand{\QQ}[1]{{\mathcal Q}^{(#1)}_{#1,\n,\p}}
\newcommand{\G}[1]{G^{(#1)}_{\n,\p}}
\renewcommand{\H}[1]{H^{(#1)}_{\n,\p}}
\newcommand{\I}[1]{I^{(#1)}_{\n,\p}}
\newcommand{\J}[1]{J^{(#1)}_{\n,\p}}
\renewcommand{\a}[1]{\alpha^{(#1)}_{\n,\p}}
\renewcommand{\b}[1]{\beta^{(#1)}_{\n,\p}}
\renewcommand{\c}[1]{\gamma^{(#1)}_{\n,\p}}
\renewcommand{\d}[1]{d^{(#1)}_{\n,\p}}

\newcommand{\p}{\mathbf{p}}
\newcommand{\q}{\mathbf{q}}
\renewcommand{\r}{\mathbf{r}}
\newcommand{\n}{\mathbf{n}}
\newcommand{\0}{\mathbf{0}}
\numberwithin{equation}{section}

\begin{document}
\title[Expansions of fundamental solution]{Expansions for a fundamental solution of Laplace's equation on $\R^3$ in 5-cyclidic harmonics}

\author{Howard S.~Cohl${}^{1}$ and Hans Volkmer${}^2$}

\address{$^1$Applied and Computational Mathematics Division,
National Institute of Standards and Technology, Gaithersburg, MD
20899-8910, USA}
\address{$^2$Department of Mathematical Sciences, University of Wisconsin--Milwaukee,
P.~O. Box 413, Milwaukee, WI 53201, USA}

\begin{abstract}
We derive eigenfunction expansions for a fundamental solution of Laplace's
equation in three-dimensional Euclidean space in 5-cyclidic coordinates.
There are three such expansions in terms of internal and external 5-cyclidic harmonics of 
first, second and third kind. The internal and external 5-cyclidic harmonics are expressed by 
solutions of a Fuchsian differential equation with five regular singular points. 
\end{abstract}

\maketitle

\section{Introduction}\label{intro}
Expansions for a fundamental solution of
Laplace's equation on $\R^3$ in terms of solutions found by the method 
of separation of variables in a suitable curvilinear coordinate system
are known for a long time.
For example, when we choose spherical coordinates, we obtain the 
well-known expansion \cite{MorseFesh}
\begin{equation}\label{1:exp1}
\frac{1}{\|\r-\r'\|}=\sum_{\ell=0}^\infty \frac{r^\ell}{(r')^{\ell+1}}\sum_{m=-\ell}^{\ell} \frac{(\ell-m)!}{(\ell+m)!}{\mathsf P}_\ell^m(\cos\theta)
{\mathsf P}_\ell^m(\cos\theta')e^{im(\phi-\phi')},
\end{equation} 
where $\|\r\|<\|\r'\|$ ($\|\r\|$ denotes the Euclidian norm of $\r\in\R^3$), and $r,\theta,\phi$, $r',\theta',\phi'$ are the spherical coordinates of $\r$ and $\r'$, respectively. 
The expansion \eqref{1:exp1} contains the Ferrers function of the first kind
(associated Legendre function of the first kind on-the-cut)
${\mathsf P}_\ell^m$ \cite[(14.3.1)]{NIST}.
We may write expansion \eqref{1:exp1} in the more concise form 
\begin{equation}\label{1:exp2}
\frac{1}{\|\r-\r'\|}=\sum_{\ell=0}^\infty\sum_{m=-\ell}^\ell G_\ell^m(\r)\overline{H_\ell^m(\r')},
\end{equation}
where $G_\ell^m:\R^3\to\C$ is the internal spherical harmonic 
\begin{equation}\label{1:G}
 G_\ell^m(\r):=\left(\frac{(\ell-m)!}{(\ell+m)!}\right)^{1/2}r^\ell {\mathsf P}_\ell^m(\cos\theta)e^{im\phi},
\end{equation}
and $H_\ell^m:\R^3\to\C$ is the external spherical harmonic
\begin{equation}\label{1:H}
 H_\ell^m(\r'):=\left(\frac{(\ell-m)!}{(\ell+m)!}\right)^{1/2}(r')^{-\ell-1}{\mathsf P}_\ell^m(\cos\theta')e^{im\phi'}.
\end{equation} 
In this paper we derive expansions analogous to \eqref{1:exp2} for the 5-cyclidic coordinate system \cite[(6.24)]{Miller} in place of spherical coordinates.
The coordinate surfaces of 5-cyclidic coordinates are triply-orthogonal confocal cyclides. There are three kinds of internal and external 5-cyclidic harmonics, one for each family of coordinate surfaces, and three corresponding expansions.
The authors already introduced internal 5-cyclidic harmonics in \cite{CohlVolkcyclide1}. As far as we know, the definition of external 5-cyclidic harmonics and 
the expansions analogous to \eqref{1:exp2} are given in this paper for the first time. We also derive some needed additional properties  of internal 5-cyclidic harmonics.
In the definitions of internal and external spherical harmonics \eqref{1:G}, \eqref{1:H} there appear only the associated Legendre functions apart from elementary functions.
In the case of 5-cyclidic coordinates the definition of internal and external harmonics requires solutions of a Fuchsian differential equation
with 5 regular singularities. The particular solutions of interest are eigenfunctions of 
two-parameter Sturm-Liouville eigenvalue problems; see \cite{CohlVolkcyclide1}.  

In Maxime B\^{o}cher's 1891 dissertation,
{\it Ueber die Reihenentwickelungen der Potentialtheorie} \cite{Bocher},
it was shown that the 3-variable Laplace equation can be solved using
separation of variables in seventeen conformally distinct quadric and cyclidic
coordinate systems. These coordinates have coordinate surfaces which are zero sets
for polynomials in $x,y,z$ with degree at most two and four respectively.  The
Helmholtz equation
on $\R^3$ admits simply separable solutions in the same eleven quadric coordinate
systems that the Laplace equation admits separable solutions \cite{EisenStac}.
The Laplace equation also admits ${R}$-separable solutions in an additional six
conformally distinct coordinate systems \cite[Table 17, page 210]{Miller}.
Unlike the Laplace equation, the
Helmholtz equation does not admit solutions via ${R}$-separation of variables.
The appearance
of ${R}$-separation is intrinsic to the
existence of conformal symmetries for a linear partial differential
equation (see Boyer, Kalnins \& Miller (1976) \cite{BoyKalMil}),
i.e.~dilatations, special conformal transformations, inversions and reflections.
The theory of separation of variables from a Lie group theoretic viewpoint has
been treated in Miller (1977) \cite{Miller}.  In Miller's book,
separation of variables for the Laplace equation on $\R^3$ was treated and
the general asymmetric ${R}$-separable 5-cyclidic coordinate system was
introduced (see \cite[Table 17, System 12]{Miller}).  In regard to this coordinate
system, and the corresponding separable harmonic solutions, Miller
indicates that ``Very little is known about the solutions.''

To the authors' knowledge, eigenfunction expansions for the fundamental solution (the $1/r$ potential)
have been obtained for the following coordinate systems. 
See \cite{CohlVolkmer,Hob,MacRobert47,MorseFesh} for expansions
in spherical, circular/parabolic/elliptic cylinder, oblate/prolate
spheroidal, parabolic, bi-spherical and toroidal coordinates.  
The expansion in confocal ellipsoidal coordinates is treated in \cite{Volkmer,Heine}. 
This paper is a stepping-stone for derivations of eigenfunction expansions for
a fundamental solution of Laplace's equation in coordinate systems where these
expansions are not known such as paraboloidal, flat-ring cyclide,
flat-disk cyclidic, bi-cyclide, cap-cyclide and
3-cyclide \cite[Table 17, System 13]{Miller}) coordinates.

The eigenfunction
expansions are often connected
with integral identities (such as the integral of Lipschitz
\cite[Section 13.2]{Watson} and the Lipschitz-Hankel integral
\cite[Section 13.21]{Watson} which appear in cylindrical coordinates),
addition theorems (such as Neumann's and Graf's generalization
of Neumann's addition theorem \cite[Section 11.1, Section 11.3]{Watson}
which appear in cylindrical coordinates and the addition theorem for
spherical harmonics \cite{WenAvery} which appears in spherical coordinates),
generating functions for orthogonal polynomials (such as the generating
function for Legendre polynomials \cite[(18.12.4)]{NIST} which appears in
spherical coordinates), and special function expansion identities (such as Heine's
reciprocal square root identity \cite[(3.11)]{CohlDominici} which appears in circular
cylindrical coordinates). In this setting, one may perform eigenfunction expansions
for a fundamental solution of Laplace's equation in alternative separable
coordinate systems to obtain new special function summation and integration identities
which often have interesting geometrical interpretations (see for
instance \cite{Cohlerratum12, CRTB,CTRS}).
Eigenfunction expansions for fundamental solutions of elliptic partial
differential equations have been extended to more general separable linear partial
differential equations \cite{Cohl12pow} and to partial differential equations on
Riemannian manifolds of constant curvature \cite{CohlKalII}.  

The outline of this paper is as follows.
The 5-cyclidic coordinate system $s_1,s_2,s_3$ is discussed in Section 2. 
In Section 3, we consider
internal and external 5-cyclidic harmonics of the second kind which are related to the coordinate surfaces $s_2={\rm const}$. We start with functions of the second kind because
they are slightly easier to treat than the harmonics of the first and third kind related to the coordinate surfaces $s_1={\rm const}$, $s_3={\rm const}$, respectively.
In Section 4, as one of our main results, we obtain the expansion of the fundamental solution of Laplace's equation in terms of internal and external 5-cyclidic harmonics of the second kind. The proof is based on $(a)$ an integral representation of the external harmonics in terms of internal harmonics given
in Section 4, and $(b)$
the completeness property of internal harmonics obtained in \cite{CohlVolkcyclide1}. In Sections 5,6 we treat 5-cyclidic harmonics of the first kind. In Sections 7,8 we treat 5-cyclidic harmonics of the third kind.

\section{5-cyclidic coordinates}\label{coordinates}
We work on $\R^3$ with Cartesian coordinates $x,y,z$, and we use
the notations $\r=(x,y,z)$ and $\|\r\|=(x^2+y^2+z^2)^{1/2}$.
Fix $a_0<a_1<a_2<a_3$.
The 5-cyclidic coordinates of a point $\r\in\R^3$ are the solutions $s=s_1,s_2,s_3$ of the equation
\begin{equation}\label{2:surface}
\frac{(\|\r\|^2-1)^2}{s-a_0}+ \frac{4x^2}{s-a_1}+\frac{4y^2}{s-a_2}+ \frac{4z^2}{s-a_3}=0
\end{equation}
(strictly speaking, this equation is multiplied by the common denominator of the left-hand side),
where
\[ a_0\le s_1\le a_1\le s_2\le a_2\le s_3\le a_3 ;\]
see \cite[Section 4]{CohlVolkcyclide1}.
On the set
\begin{equation}\label{2:R}
R:=\{\r:  x,y,z>0,\, \|\r\|<1\},
\end{equation}
the map
$(x,y,z)\in R\mapsto (s_1,s_2,s_3)\in(a_0,a_1) \times (a_1,a_2)\times(a_2,a_3)$ is bijective.
The inverse map is given by
\begin{equation}\label{2:coord1}
x=\frac{x_1}{1+x_0},\quad y=\frac{x_2}{1+x_0},\quad z=\frac{x_3}{1+x_0},\\
\end{equation}
where
\begin{equation}\label{2:coord2}
x_j^2=\frac{\prod_{i=1}^3 (s_i-a_j)}{\prod_{j\ne i=0}^3 (a_i-a_j)},\quad x_j>0 .
\end{equation}
We note that each $s_i$ is a continuous function on $\R^3$. Of particular interest are the sets
\begin{eqnarray*}
 A_1&:=&\{\r: s_1=s_2\}=\{(0,y,z): g_1(y,z)=0\},\\
 A_2&:=&\{\r: s_2=s_3\}=\{(x,0,z): g_2(x,z)=0\},
\end{eqnarray*}
where
\begin{eqnarray*}
 g_1(y,z)&:=&\frac{(y^2+z^2-1)^2}{a_1-a_0}+\frac{4y^2}{a_1-a_2}+\frac{4z^2}{a_1-a_3},\\
 g_2(x,z)&:=&\frac{(x^2+z^2-1)^2}{a_2-a_0}+\frac{4x^2}{a_2-a_1}+\frac{4z^2}{a_2-a_3}.
\end{eqnarray*}
Each set $A_1, A_2$ consists of two closed curves; see Figures \ref{2:fig1}, \ref{2:fig2}.
The function $s_1$ is (real-)analytic on $\R^3\setminus A_1$, $s_2$ is analytic on $\R^3\setminus(A_1\cup A_2)$, and $s_3$ is analytic on
$\R^3\setminus A_2$.
We will also encounter the sets
\begin{eqnarray*}
K_1&:=&\{\r: \|\r\|<1, s_1=a_1\}=\{(0,y,z): y^2+z^2<1, g_1(y,z)\ge 0\},\\
L_1&:=&\{\r: s_2=a_1\}=\{(0,y,z): g_1(y,z)\le 0\},\\
M_1&:=&\{\r: \|\r\|>1, s_1=a_1\}=\{(0,y,z): y^2+z^2>1, g_1(y,z)\ge 0\},\\
K_2&:=&\{\r: z>0, s_3=a_2\}=\{(x,0,z): z>0, g_2(x,z)\le 0\},\\
L_2&:=&\{\r: s_2=a_2\}=\{(x,0,z): g_2(x,z)\ge 0\} ,\\
M_2&:=&\{\r: z<0, s_3=a_2\}=\{(x,0,z): z<0, g_2(x,z)\le 0\}.
\end{eqnarray*}
The sets $A_1,K_1,L_1,M_1$ are subsets of the plane $x=0$, and $A_2,K_2,L_2,M_2$ are subsets of the plane $y=0$; see Figures \ref{2:fig1}, \ref{2:fig2}.

\begin{figure}[h]
\begin{center}
\includegraphics[height=10cm,width=10cm]{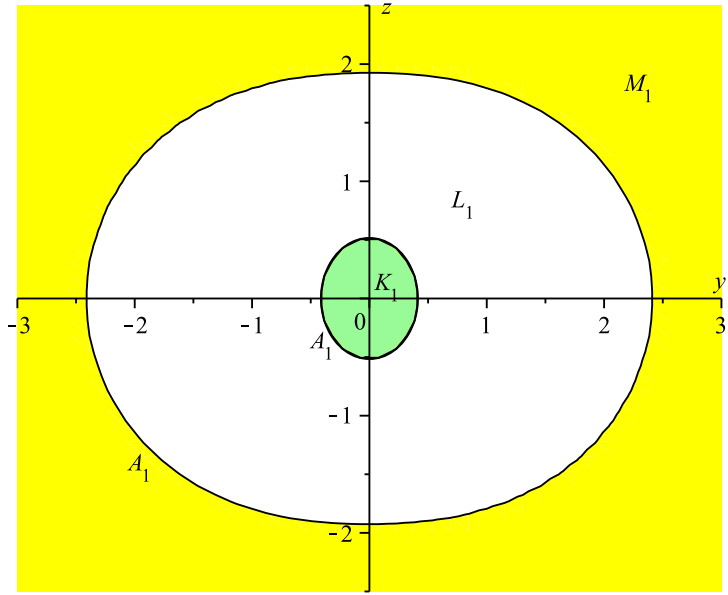}
\end{center}
\vspace*{-5mm}
\caption{Curves $A_1$ and regions $K_1,L_1,M_1$ for $a_j=j$.}
\label{2:fig1}
\end{figure}

\begin{figure}[h]
\begin{center}
\includegraphics[height=10cm,width=10cm]{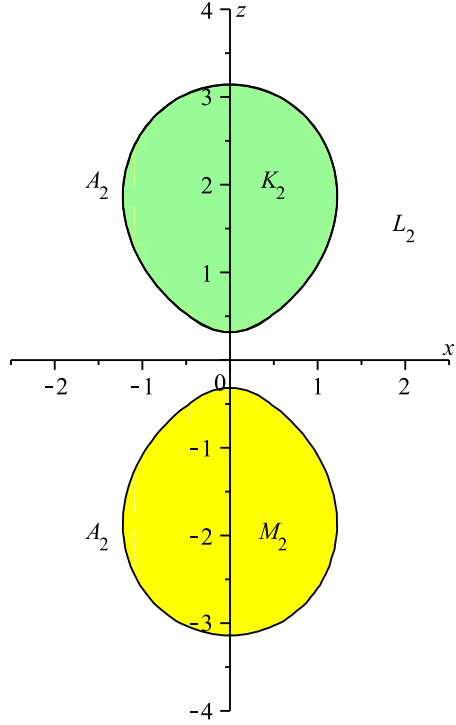}
\end{center}
\vspace*{-5mm}
\caption{Curves $A_2$ and regions $K_2,L_2,M_2$ for $a_j=j$.}
\label{2:fig2}
\end{figure}

We denote the inversion at the unit sphere on $\R^3$ by
\begin{equation}\label{2:inversion}
\sigma_0(\r):=\|\r\|^{-2} \r ,
\end{equation}
and the reflections at the coordinate planes by
\begin{equation}\label{2:reflections}
\sigma_1(x,y,z):=(-x,y,z),\,\sigma_2(x,y,z):=(x,-y,z),\,\sigma_3(x,y,z):=(x,y,-z) .
\end{equation}
We note that the functions $s_1,s_2,s_3$ are invariant under $\sigma_j$, $j=0,1,2,3$.

We define auxiliary functions $\chi_j:\R^3\to\R$, $j=0,1,2,3$, by
\begin{eqnarray*}
\chi_0(\r)&:=&{\rm sgn}(1-\|\r\|) (s_1-a_0)^{1/2},\\
\chi_1(\r)&:=&{\rm sgn}(x) ((s_2-a_1)(a_1-s_1))^{1/2},\\
\chi_2(\r)&:=&{\rm sgn}(y) ((s_3-a_2)(a_2-s_2))^{1/2},\\
\chi_3(\r)&:=&{\rm sgn}(z) (a_3-s_3)^{1/2}.
\end{eqnarray*}

\begin{lemma}\label{2:l1}
The functions $\chi_j$, $j=0,1,2,3$, are continuous on $\R^3$. $\chi_0, \chi_2$ are analytic on $\R^3\setminus A_1$, and  $\chi_1,\chi_3$ are analytic on $\R^3\setminus A_2$. Moreover,
\begin{equation}\label{2:chisigma}
 \chi_j\circ\sigma_i =\begin{cases} \chi_j & \text{if $i\ne j$},\\ -\chi_j & \text{if $i=j$.}\end{cases}
\end{equation}
\end{lemma}
\begin{proof}
Consider first $\chi_3$. The function $s_3$ is continuous, and $s_3=a_3$ if and only if $z=0$. Therefore, $\chi_3$ is continuous. In order to prove that $\chi_3$ is analytic on $\R^3\setminus A_2$,
it is enough to show that $\chi_3$ is analytic at every point of the plane $z=0$.
Let $\r_0=(x_0,y_0,0)$. There is $\epsilon\in(0,1)$ such that $s_3\ne a_2$ for $\r\in B_\epsilon(\r_0)=\{\r:  \|\r-\r_0\|<\epsilon\}$.
Then \eqref{2:surface} with $s=s_3$ implies
\[ a_3-s_3= \frac{4z^2}{f(\r)}\quad\text{for $\r\in B_\epsilon(\r_0)$}, \]
where
\[ f(\r):=\frac{(\|\r\|^2-1)^2}{s_3-a_0}+\frac{4x^2}{s_3-a_1}+\frac{4y^2}{s_3-a_2} \]
is positive and analytic on $B_\epsilon(\r_0)$.
Therefore, we obtain
\[ \chi_3(\r)=\frac{2z}{(f(\r))^{1/2}}\quad\text{for $\r\in B_\epsilon(\r_0)$}, \]
and this shows that $\chi_3$ is analytic at $\r_0$. $\chi_0$ is treated similarly.

Consider next $\chi_2$.  The functions $s_2, s_3$ are continuous, and $(a_2-s_2)(s_3-a_2)=0$ if and only if $y=0$. Thus $\chi_2$ is continuous.
In order to prove that $\chi_2$ is analytic on $\R^3\setminus A_1$, it is enough to show that $\chi_2$ is analytic at all points of the plane $y=0$ which do not lie in $A_1$.
Suppose first $\r_0=(x_0,0,z_0)\in (K_2\cup M_2)\setminus A_2$. There is $\epsilon>0$ such that $s_3\ne a_3$ and $s_2\ne a_2$ for $\r\in B_\epsilon(\r_0)$.
Then, by \eqref{2:surface} with $s=s_3$, we obtain
\[ s_3-a_2=\frac{4y^2}{g(\r)},\]
where
\[ g(\r):=-\frac{(\|\r\|^2-1)^2}{s_3-a_0}-\frac{4x^2}{s_3-a_1}-\frac{4z^2}{s_3-a_3} \]
is analytic on $B_\epsilon(\r_0)$. Since $g(\r_0)=-g_2(x_0,z_0)>0$, $g$ is also positive on $B_\epsilon(\r_0)$ for sufficiently small $\epsilon>0$.
Then
\[ \chi_2(\r)=(a_2-s_2)^{1/2} \frac{2y}{(g(\r))^{1/2}} \quad\text{for $\r\in B_\epsilon(\r_0)$.}
\]
This shows that $\chi_2$ is analytic at $\r_0$ provided that $\r_0\notin A_1$.
In a similar way, by using \eqref{2:surface} with $s=s_2$, we show that $\chi_2$ is analytic at all points $\r_0\in L_2\setminus A_2$.
Finally, by subtracting equations \eqref{2:surface} with $s=s_2,s_3$ from each other, we show that $\chi_2$ is analytic at all points $\r_0\in A_2$.
$\chi_1$ is treated similarly.

The symmetries \eqref{2:chisigma} follow from the definition of $\chi_j$.
\end{proof}

Solving the Laplace equation
\begin{equation}\label{2:Laplace}
\frac{\partial^2 u}{\partial x^2}+\frac{\partial^2 u}{\partial y^2}+\frac{\partial^2 u}{\partial z^2} =0
\end{equation}
by the method of separation of variables, we find  solutions
\begin{equation}\label{2:sepsol}
u(\r):=(\|\r\|^2+1)^{-1/2}w_1(s_1)w_2(s_2)w_3(s_3),\quad s_i\in(a_{i-1},a_i).
\end{equation}
Each function $w=w_1,w_2,w_3$ satisfies the Fuchsian equation
\begin{equation}\label{2:Fuchs}
\prod_{j=0}^3 (s-a_j)\left[w''+\frac12 \sum_{j=0}^3 \frac{1}{s-a_j} w'\right] +\left(\frac{3}{16}s^2+\lambda_1s+\lambda_2\right)w =0 ,
\end{equation}
where $\lambda_1,\lambda_2$ are separation constants; see \cite{CohlVolkcyclide1}.
This equation has five regular singularities at $a_0,a_1,a_2,a_3,\infty$. The exponents at each finite singularity are $0$ or $\frac12$.

The function $u(\r)$ defined in \eqref{2:sepsol} is harmonic for all choices of solutions $w_i$ to \eqref{2:Fuchs}. However, it is harmonic only in the open set obtained from $\R^3$ by removing the
coordinate planes $x=0, y=0, z=0$ and the unit sphere $\|\r\|=1$.
In order to obtain globally defined harmonic functions we have to select the Frobenius solutions $w$ at the finite singularities, that is, solutions that are either analytic at $a_j$ or of the form
$(s-a_j)^{1/2} g(s)$ with $g(s)$ analytic at $s=a_j$.
It is impossible to choose the parameters $\lambda_1,\lambda_2$ in such a way that each solution $w_i$, $i=1,2,3$, is a nontrivial Frobenius solution belonging to either one of the exponents
$0$ or $\frac12$ at both end points $a_{i-1}$, $a_i$. If this were possible \eqref{2:sepsol} would define a function which is harmonic in the whole space $\R^3$ (as we see later) and converges to $0$ as $\|\r\|\to\infty$. But such a function would have to be identically zero. However, as shown in \cite{CohlVolkcyclide1}, we can determine special values of $\lambda_1,\lambda_2$ (eigenvalues) such that
two solutions (either (1) $w_2$, $w_3$, or (2) $w_1$, $w_3$, or (3) $w_1$, $w_2$) are nontrivial Frobenius solution at both end points simultaneously.
These cases lead to 5-cyclidic harmonics of the first, second and third kind. If the remaining function $w_i$ in case $(i)$ is chosen appropriately, we obtain internal
or external 5-cyclidic harmonics.

\section{5-cyclidic harmonics of the second kind}\label{harmonics2}

In \cite[Section VII]{CohlVolkcyclide1} we introduced special solutions $w_i(s_i)=\E{2}{i}(s_i)$ to equation \eqref{2:Fuchs} for eigenvalues $\lambda_j=\lambda^{(2)}_{j,\n,\p}$, $j=1,2$, for every $\n\in\N_0^2$, $\p=(p_0,p_1,p_2,p_3)\in\{0,1\}^4$.
If $\n=(n_1,n_3)$ then $n_i$ denotes the number of zeros of $\E{2}{i}$ in $(a_{i-1},a_i)$ for $i=1,3$.
The subscript $p_j$ describes the behavior of the solutions at the endpoint $a_j$:
We have
\[ \E{2}{i}(s_i)=(s_i-a_{i-1})^{p_{i-1}/2}(a_i-s_i)^{p_i/2} \EE{2}{i}(s_i),\quad s_i\in(a_{i-1},a_i),
\]
where $\EE{2}{1}$ is analytic on $[a_0,a_1]$, $\EE{2}{2}$ is analytic on $[a_1,a_2)$  (but not at $a_2$), and
$\EE{2}{3}$ is analytic on $[a_2,a_3]$.

According to \eqref{2:sepsol} the function
\begin{equation}\label{3:G0}
\G{2}(\r):=(\|\r\|^2+1)^{-1/2} \E{2}{1}(s_1)\E{2}{2}(s_2)\E{2}{3}(s_3),\quad \r\in R,
\end{equation}
is harmonic on $R$.
In order to analytically extend $\G{2}$ we use the functions $\chi_j$ introduced in Section 2.
We set
\begin{equation}\label{3:G}
\G{2}(\r):= (\|\r\|^2+1)^{-1/2} \prod_{j=0}^3 (\chi_j(\r))^{p_j} \prod_{i=1}^3 \EE{2}{i}(s_i) \quad \text{if $s_2\ne a_2$}
\end{equation}
which is consistent with \eqref{3:G0}. The condition $s_2\ne a_2$ is equivalent to $\r\in\R^3\setminus L_2$. We call $\G{2}$ an internal 5-cyclidic harmonic of the second kind.

\begin{thm}\label{3:t1}
Let $\n\in\N_0^2$ and $\p\in\{0,1\}^4$. Then $\G{2}$ is harmonic on $\R^3\setminus L_2$.
Moreover,
\begin{equation}\label{3:sym1}
 \G{2}(\sigma_j(\r))=(-1)^{p_j}  \G{2}(\r)\quad\text{for $j=1,2,3$} ,
\end{equation}
and
\begin{equation}\label{3:sym2}
  \G{2}(\sigma_0(\r))=(-1)^{p_0}  \|\r\|\G{2}(\r) .
\end{equation}
\end{thm}
\begin{proof}
By \eqref{3:G} and Lemma \ref{2:l1}, $\G{2}$ is a composition of continuous functions, and thus it is continuous on $\R^3\setminus L_2$.
As a composition of analytic functions, $\G{2}$ is analytic and thus harmonic on $\R^3\setminus (A_1\cup L_2)$.
The set $A_1$ is a removable line singularity of $\G{2}$. This can be seen in two different ways.
1) We may appeal to the general theory of harmonic functions. $A_1$ is a polar set, and we may apply \cite[Cor.\ 5.2.3]{AG}.
2) We can show directly that $\G{2}$ is analytic at each point of $A_1$ by the method used in the proof of \cite[Lemma 6.1]{CohlVolkcyclide1}.
For example, take the simplest case $\p=(0,0,0,0)$. Then \eqref{3:G0} holds for all $\r\in \R^3\setminus L_2$, and the product $\E{2}{1}(s_1)\E{2}{2}(s_2)$ is  analytic at each point of $A_1$.
This is because $\E{2}{1}(s)$ and $\E{2}{2}(s)$ are analytic extensions of each other, 
and $s_1,$ $s_2$ enter symmetrically. Note that $s_1s_2$ and $s_1+s_2$ are analytic at each point of $A_1$ although $s_1$, $s_2$ are not analytic there.

The symmetry properties of $\G{2}$ also follow from \eqref{3:G} and Lemma \ref{2:l1}.
\end{proof}

If $U(\r)$ is a harmonic function then its Kelvin transformation
\[ V(\r)=\|\r\|^{-1}U(\sigma_0(\r))\]
is also harmonic \cite[page 232]{Kellogg}.
Equation \eqref{3:sym2} states that $\G{2}$ is invariant or changes sign under the Kelvin transformation if $p_0=0$ or $p_0=1$, respectively.
We see that $L_2$ is a ``surface singularity'' of $\G{2}$ which is not removable (it is not a polar set).
In fact, $\G{2}$ cannot be harmonic on $\R^3$ because it would be identically zero otherwise.

Let $\F{2}$ be the Frobenius solution to the Fuchsian equation \eqref{2:Fuchs} (with $\lambda_j=\lambda^{(2)}_{j,\n,\p}$)   on $(a_1,a_2)$
belonging to the exponent $\frac{p_2}{2}$ at $s_2=a_2$, uniquely
determined by the Wronskian condition
\begin{equation}\label{3:Wronskian}
\omega(s)\left(\E{2}{2}(s_2)\frac{d}{ds_2}\F{2}(s_2)-\F{2}(s_2)\frac{d}{ds_2}\E{2}{2}(s_2)\right) =1,
\end{equation}
where
\begin{equation}\label{3:omega}
 \omega(s):=\left|(s-a_0)(s-a_1)(s-a_2)(s-a_3)\right|^{1/2} .
\end{equation}
This definition is possible because we know that $\E{2}{2}(s_2)$ is not a Frobenius solution belonging to the exponent $\frac{p_2}2$ at $s_2=a_2$.
Now we define external 5-cyclidic harmonics of the second kind by
\begin{equation}\label{3:H0}
\H{2}(\r):=(\|\r\|^2+1)^{-1/2}\E{2}{1}(s_1)\F{2}(s_2)\E{2}{3}(s_3),\quad \r\in R .
\end{equation}
In order to analytically extend $\H{2}$ we write
\[ \F{2}(s_2)=(s_2-a_1)^{p_1/2}(a_2-s_2)^{p_2/2} \FF{2}(s_2),\quad s_2\in(a_1,a_2),
\]
where $\FF{2}$ is analytic on $(a_1,a_2]$ (but not at $a_1$).
Then we define
\begin{equation}\label{3:H}
\H{2}(\r):= (\|\r\|^2+1)^{-1/2} \prod_{j=0}^3 (\chi_j(\r))^{p_j} \EE{2}{1}(s_1)\FF{2}(s_2)\EE{2}{3}(s_3) \quad \text{if $s_2\ne a_1$}.
\end{equation}
The condition $s_2\ne a_1$ is equivalent to $\r\in \R^3\setminus L_1$.

\begin{thm}\label{3:t2}
Let $\n\in\N_0^2$ and $\p\in\{0,1\}^4$. Then $\H{2}$
is harmonic on $\R^3\setminus L_1$.
The functions $\H{2}$ share the symmetries \eqref{3:sym1}, \eqref{3:sym2} with $\G{2}$.
Moreover,
\begin{equation}\label{3:est1}
\H{2}(\r) =O(\|\r\|^{-1})\quad\text{as $\|\r\|\to\infty$},
\end{equation}
and
\begin{equation}\label{3:est2}
\|\nabla \H{2}(\r)\| =O(\|\r\|^{-2})\quad\text{as $\|\r\|\to\infty$} .
\end{equation}
\end{thm}
\begin{proof}
The proof of analyticity and symmetry of $\H{2}$ is similar to that given for $\G{2}$ in Theorem \ref{3:t1}, and is omitted.
Estimates \eqref{3:est1} and \eqref{3:est2} follow easily from the observation that the Kelvin transformation of $\H{2}$ is $\pm \H{2}$ which is analytic at $\0\notin L_1$.
\end{proof}

\section{Expansion of the reciprocal distance in 5-cyclidic harmonics of second kind}\label{expansion2}

For given $d_2\in(a_1,a_2)$ we consider the ``5-cyclidic ring''
\begin{equation}\label{4:ring1}
D_2:=\{\r\in\R^3: s_2<d_2\},
\end{equation}or, equivalently,
\begin{equation}\label{4:ring2}
D_2=\{\r: \frac{(\|\r\|^2-1)^2}{d_2-a_0}+\frac{4x^2}{d_2-a_1}+\frac{4y^2}{d_2-a_2}+\frac{4z^2}{d_2-a_3}<0 \}.
\end{equation}
Note that each internal 5-cyclidic harmonic $\G{2}$ is harmonic in $D_2$ (and on its boundary), and each external 5-cyclidic harmonic is harmonic on $\R^3\setminus D_2$ (and on its boundary).

We represent external harmonics in terms of internal harmonics by a surface integral over the boundary $\partial D_2$ of the ring $D_2$ as follows.
\begin{thm}\label{4:t1}
Let $d_2\in(a_1,a_2)$, $\n\in\N_0^2$, $\p\in\{0,1\}^4$. Then
\begin{equation}\label{4:integralformula}
\H{2}(\r')=\frac{1}{4\pi\omega(d_2)\{\E{2}{2}(d_2)\}^2}
\int_{\partial D_2} \frac{\G{2}(\r)}{h_2(\r)\|\r-\r'\|}\, dS(\r)
\end{equation}
for all $\r'\in\R^3\setminus\bar D_2$.
The scale factor $h_2$ is given by
\begin{equation}\label{4:h2}
 16 \{h_2(\r)\}^2=\frac{(\|\r\|^2-1)^2}{(d_2-a_0)^2}+ \frac{4x^2}{(d_2-a_1)^2}+\frac{4y^2}{(d_2-a_2)^2}+ \frac{4z^2}{(d_2-a_3)^2}.
\end{equation}
\end{thm}
\begin{proof}
Let $D$ be an open bounded subset of $\R^3$ with smooth boundary.
For $u,v\in C^2(\bar D)$, Green's formula states that
\begin{equation}\label{4:Greenformula}
\int_D(u\Delta v-v\Delta u)\,d\r = \int_{\partial D} \left(u\frac{\partial v}{\partial\nu}-v\frac{\partial u}{\partial\nu}\right)\,dS,
\end{equation}
where $\frac{\partial u}{\partial\nu}$ is the outward normal derivative of $u$ on the boundary $\partial D$ of $D$.

We apply \eqref{4:Greenformula} to the domain $D=D_2$, and functions $u=G=\G{2}$, $v(\r)=\frac{1}{4\pi\|\r-\r'\|}$.
Since $u,v$ are harmonic on an open set containing $\bar D_2$ we obtain
\begin{equation}\label{4:eq1}
0=\int_{\partial D_2} \left(G \frac{\partial v}{\partial\nu}-v \frac{\partial G}{\partial\nu}\right) dS .
\end{equation}

We now use \eqref{4:Greenformula} a second time. We choose $R>0$ so large that the ball $B_R(\0)$ contains $\r'$ and
$\bar D_2$. Then we take $D=B_R(\0)-\bar D_2-B_\epsilon(\r')$ with small radius $\epsilon>0$.
Take $u=H=\H{2}$ and $v$ as before. Note that $u, v$ are harmonic on an open set containing $\bar D$.
By a standard argument \cite[Theorem 1, page 109]{McOwen}, taking the limit $\epsilon\to 0$, we obtain
\begin{equation}\label{4:eq2}
H(\r')=\int_{\partial B_R(\0)} \left(H\frac{\partial v}{\partial\nu}-v \frac{\partial H}{\partial\nu}\right) dS-
\int_{\partial D_2} \left(H\frac{\partial v}{\partial\nu}-v \frac{\partial H}{\partial\nu}\right) dS,
\end{equation}
where, in the second integral, $\frac{\partial}{\partial\nu}$ denotes the same derivative as in \eqref{4:eq1}.
The first integral in \eqref{4:eq2} tends to $0$ as $R\to\infty$ by \eqref{3:est1}, \eqref{3:est2}.
Therefore,
\begin{equation}\label{4:eq3}
H(\r')=-\int_{\partial D_2} \left(H\frac{\partial v}{\partial\nu}-v \frac{\partial H}{\partial\nu}\right) dS.
\end{equation}

We now multiply \eqref{4:eq1} by $F_2(d_2)$, $F_2:=\F{2}$, then multiply \eqref{4:eq3} by $E_2(d_2)$, $E_i:=\E{2}{i}$, and add these equations.
By \eqref{3:G0} and \eqref{3:H0} we have
\[ F_2(d_2)G(\r)=E_2(d_2)H(\r),\quad \r\in\partial D_2,\]
first for $\r\in \partial D_2\cap R$ but then for all $\r\in \partial D_2$ by shared symmetries \eqref{3:sym1}, \eqref{3:sym2} of $G,H$.
Therefore, we find
\begin{equation}\label{4:eq4}
E_2(d_2)H(\r')=\int_{\partial D_2} v\left(E_2(d_2)\frac{\partial H}{\partial \nu}-
F_2(d_2)\frac{\partial G}{\partial\nu}\right) dS .
\end{equation}
The normal derivative and the derivative with respect to $s_2$ are related by
\[ \frac{\partial}{\partial \nu}= \frac{1}{h_2} \frac{\partial}{\partial s_2},\]
where $h_2$ is the scale factor of the 5-cyclidic coordinate $s_2$ given by \eqref{4:h2}; see \cite[(22)]{CohlVolkcyclide1}.
Let $\r\in \partial D_2\cap R$ with 5-cyclidic
coordinates $s_1,s_2=d_2,s_3$.
Then
\begin{eqnarray*}
&&\hspace{-0.2cm}\left(E_2(d_2)\frac{\partial H}{\partial \nu}-
F_2(d_2)\frac{\partial G}{\partial\nu}\right)(\r) \\
&&\hspace{1.3cm}= E_2(d_2)\frac{\partial (\|\r\|^2+1)^{-1/2}}{\partial\nu}E_1(s_1)F_2(d_2)E_3(s_3)\\
&&\hspace{2.3cm} +E_2(d_2)(\|\r\|^2+1)^{-1/2}h_2^{-1} E_1(s_1)F_2'(d_2)E_3(s_3) \\
&&\hspace{2.3cm} -F_2(d_2)\frac{\partial (\|\r\|^2+1)^{-1/2}}{\partial\nu} E_1(s_1)E_2(d_2)E_3(s_3)\\
&&\hspace{2.3cm} -F_2(d_2)(\|\r\|^2+1)^{-1/2} h_2^{-1}E_1(s_1)E_2'(d_2)E_3(s_3)\\
&&\hspace{1.3cm}= h_2^{-1} (\|\r\|^2+1)^{-1/2}E_1(s_1)\left\{E_2(d_2)F_2'(d_2)
-E_2'(d_2)F_2(d_2)\right\} E_3(s_3) .
\end{eqnarray*}
We now use \eqref{3:Wronskian} and obtain
\begin{equation}\label{4:eq5}
\left(E_2(d_2)\frac{\partial H}{\partial \nu}-
F_2(d_2)\frac{\partial G}{\partial\nu}\right)(\r)= \frac{G(\r)}{h_2(\r)\omega(d_2)E_2(d_2)},
\end{equation}
which holds for all $\r\in \partial D_2$ because $G$ and $H$ share the symmetries \eqref{3:sym1}, \eqref{3:sym2}.
When we substitute \eqref{4:eq5} in \eqref{4:eq4}
we arrive at \eqref{4:integralformula}
\end{proof}

We obtain the expansion of the reciprocal distance in 5-cyclidic harmonics.

\begin{thm}\label{4:t2}
Let $\r,\r'\in\R^3$ with 5-cyclidic coordinates $s_2, s_2'$, respectively.
If $s_2<s_2'$ then
\begin{equation}\label{4:expansion}
\frac{1}{\|\r-\r'\|}=\pi\sum_{\n\in\N_0^2} \sum_{\p\in\{0,1\}^4}  \G{2}(\r)\H{2}(\r').
\end{equation}
\end{thm}
\begin{proof}
We pick $d_2$ such that $s_2<d_2<s_2'$, and consider the domain $D_2$ defined in \eqref{4:ring1}.
The function $f(\q):=\|\q-\r'\|^{-1}$ is harmonic on an open set containing $\bar D_2$.
Therefore, by \cite[(95),(97)]{CohlVolkcyclide1}, we have
\begin{equation}\label{4:expansion2}
\frac{1}{\|\r-\r'\|}= \sum_{\n\in\N_0^2} \sum_{\p\in\{0,1\}^4} \d{2}\G{2}(\r),
\end{equation}
where
\[ \d{2}:=\frac{1}{4\omega(d_2)\{\E{2}{2}(d_2)\}^2} \int_{\partial D_2} \frac{\G{2}(\q)}{h_2(\q)\|\q-\r'\|}\,dS(\q).\]
Using Theorem \ref{4:t1}, we obtain \eqref{4:expansion}.
\end{proof}

\section{5-cyclidic harmonics of the first kind}\label{harmonics1}

In \cite[Section V]{CohlVolkcyclide1} we introduced special solutions $w_i(s_i)=\E{1}{i}(s_i)$ to equation \eqref{2:Fuchs} for eigenvalues $\lambda_j=\lambda^{(1)}_{j,\n,\p}$, $j=1,2$, for every $\n\in\N_0^2$, $\p=(p_1,p_2,p_3)\in\{0,1\}^3$.
These functions have the form
\begin{eqnarray*}
 \E{1}{1}(s_1)&=&(a_1-s_1)^{p_1/2} \EE{1}{1}(s_1), \quad s_1\in(a_0,a_1),\\
 \E{1}{i}(s_i)&=&(s_i-a_{i-1})^{p_{i-1}/2}(a_i-s_i)^{p_i/2} \EE{1}{i}(s_i),\quad s_i\in(a_{i-1},a_i), i=2,3,
\end{eqnarray*}
where $\EE{1}{1}$ is analytic on $(a_0,a_1]$ (but not at $a_0$) while $\EE{1}{i}$ is analytic on $[a_{i-1},a_i]$ for $i=2,3$.
As in \cite[Section VI]{CohlVolkcyclide1} we define the internal 5-cyclidic harmonic of the first kind by
\begin{equation}\label{5:G0}
\G{1}(\r):=(\|\r\|^2+1)^{-1/2}  \E{1}{1}(s_1)\E{1}{2}(s_2)\E{1}{3}(s_3), \quad \r\in R.
\end{equation}
According to \eqref{2:sepsol}, $\G{1}$ is a harmonic function in the region $R$.
In order to analytically extend $\G{1}$ to a larger domain of definition, some preparations are necessary.

Let $\P{1}$ be the solution to \eqref{2:Fuchs} (with $\lambda_j=\lambda^{(1)}_{j,\n,\p}$)
on $(a_0,a_1)$ belonging to the exponent $0$ at $s=a_0$ and uniquely determined by the condition $\P{1}(a_0)=1$.
We write
\[ \P{1}(s_1)=(a_1-s_1)^{p_1/2} \PP{1}(s_1),\quad s_1\in(a_0,a_1),\]
where $\PP{1}(s_1)$ is analytic on $[a_0,a_1)$.
Then using the functions $\chi_j$ from Section 2 we define
\begin{equation}\label{5:I}
\I{1}(\r):=(\|\r\|^2+1)^{-1/2}\prod_{j=1}^3 (\chi_j(\r))^{p_j} \PP{1}(s_1)\prod_{i=2}^3 \EE{1}{i}(s_i) \quad \text{if $s_1\ne a_1$} .
\end{equation}
The condition $s_1\ne a_1$ is equivalent to $\r\in \R^3\setminus(K_1\cup M_1)$; see Figure \ref{2:fig1}.

Similarly, let $\Q{1}$ be the solution to \eqref{2:Fuchs} (with $\lambda_j=\lambda^{(1)}_{j,\n,\p}$)
on $(a_0,a_1)$ belonging to the exponent $\frac12$ at $s=a_0$ and uniquely determined by the condition $\lim_{s_1\to a_0^+} \omega(s_1)\frac{d}{ds_1}\Q{1}(s_1)=1$.
We write
\[ \Q{1}(s_1)=(s_1-a_0)^{1/2}(a_1-s_1)^{p_1/2}\QQ{1}(s_1),\quad s_1\in(a_0,a_1),\]
where $\QQ{1}(s_1)$ is analytic on $[a_0,a_1)$.
Then we define
\begin{equation}\label{5:J}
\J{1}(\r):=(\|\r\|^2+1)^{-1/2}\chi_0(\r)\prod_{j=1}^3 (\chi_j(\r))^{p_j} \QQ{1}(s_1)\prod_{i=2}^3 \EE{1}{i}(s_i) \quad \text{if $s_1\ne a_1$} .
\end{equation}

\begin{lemma}\label{5:l1}
The functions $\I{1}$ and $\J{1}$ are harmonic on $\R^3\setminus(K_1\cup M_1)$. They have the symmetries
\begin{eqnarray}
 \I{1}(\sigma_0(\r))&=&\|\r\| \I{1}(\r), \label{5:sym1}\\
 \I{1}(\sigma_j(\r))&=&(-1)^{p_j} \I{1}(\r),\quad j=1,2,3, \label{5:sym2}\\
 \J{1}(\sigma_0(\r))&=&-\|\r\| \J{1}(\r), \label{5:sym3} \\
 \J{1}(\sigma_j(\r))&=&(-1)^{p_j} \J{1}(\r) ,\quad j=1,2,3. \label{5:sym4}
 \end{eqnarray}
\end{lemma}
\begin{proof}
By definition \eqref{5:I}, $\I{1}$ is a composition of continuous functions provided $s_1\ne a_1$, that is, $\I{1}$ is continuous on $\R^3\setminus(K_1\cup M_1)$.
$\I{1}$ is also a composition of analytic functions provided $s_1\ne a_1$ and $s_2\ne s_3$, that is, $\I{1}$ is analytic on $\R^3\setminus (K_1\cup M_1\cup A_2)$. Thus it is also harmonic
on $\R^3\setminus (K_1\cup M_1\cup A_2)$. By the same argument as in the proof of Theorem \ref{3:t1}, $A_2$ is a removable singularity of $\I{1}$. Thus
$\I{1}$ is harmonic on $\R^3\setminus (K_1\cup M_1)$.
The proof that $\J{1}$ is harmonic on $\R^3\setminus (K_1\cup M_1)$ is analogous.
The symmetry properties  follow from \eqref{5:I}, \eqref{5:J} and Lemma \ref{2:l1}.
\end{proof}

Since $\P{1},\Q{1}$ form a fundamental system of solutions to \eqref{2:Fuchs} (with $\lambda_j=\lambda^{(1)}_{j,\n,\p}$) on $(a_0,a_1)$, there are (nonzero) scalars $\a{1}$, $\b{1}$
such that
\[ \E{1}{1}=\a{1}\P{1}+\b{1} \Q{1} .\]
This leads us to the global definition of internal 5-cyclidic harmonics of the first kind
\begin{equation}\label{5:G}
\G{1}:=\a{1}\I{1}+\b{1}\J{1}
\end{equation}
which is consistent with \eqref{5:G0}.
We also note that, if $\|\r\|<1$ and $\r\not\in K_1$,  then \eqref{5:I}, \eqref{5:J}, \eqref{5:G} imply that
\begin{equation}\label{5:G1}
 \G{1}(\r)=(\|\r\|^2+1)^{-1/2}\prod_{j=1}^3 (\chi_j(\r))^{p_j} \prod_{i=1}^3 \EE{1}{i}(s_i) .
\end{equation}

\begin{thm}\label{5:t1}
Let $\n\in\N_0^2$ and $\p=(p_1,p_2,p_3)\in\{0,1\}^3$. Then $\G{1}$ extends continuously to a harmonic function on $\R^3\setminus M_1$.
Moreover,
\begin{equation}\label{5:sym5}
 \G{1}(\sigma_j(\r))=(-1)^{p_j}  \G{1}(\r)\quad\text{for $j=1,2,3$} .
\end{equation}
\end{thm}
\begin{proof}
By Lemma \ref{5:l1}, $\G{1}$ is harmonic on $\R^3\setminus(K_1\cup M_1)$.
If $\|\r\|<1$ we have $s_1\ne a_0$. Therefore, the right-hand side of \eqref{5:G1} is continuous on the ball $B_1(\0)$ and harmonic on $B_1(\0)\setminus (A_1\cup A_2)$.
Thus it is harmonic on $B_1(\0)$ which proves the first part of the statement of the theorem.
The symmetries follow from \eqref{5:sym2}, \eqref{5:sym4}.
\end{proof}

It will be useful to introduce another solution to \eqref{2:Fuchs} by
\begin{equation}\label{5:F}
\F{1}(s_1):= \c{1}\left(\a{1} \P{1}(s_1)-\b{1} \Q{1}(s_1)\right),\quad s_1\in (a_0,a_1) .
\end{equation}
We determine $\c{1}$ from the  Wronskian
\begin{equation}\label{5:Wronskian}
\omega(s_1)\left(\E{1}{1}(s_1)\frac{d}{ds_1}\F{1}(s_1)-\F{1}(s_1)\frac{d}{ds_1}\E{1}{1}(s_1)\right) =1
\end{equation}
which is equivalent to
\[ \c{1}=\frac{-1}{2\a{1}\b{1}} .\]
We define external 5-cyclidic harmonics of the first kind by
\begin{equation}\label{5:H}
\H{1}(\r):=\c{1}\|\r\|^{-1} \G{1}(\sigma_0(\r)) \quad \text{for $\r\in\R^3\setminus K_1$}.
\end{equation}
The reason to include the factor $\c{1}$ is that we aim for a simple form of the expansion formula \eqref{6:expansion}.
In particular, we have
\begin{equation}\label{5:H1}
 \H{1}(\r)=(\|\r\|^2+1)^{-1/2}\F{1}(s_1)\E{1}{2}(s_2)\E{1}{3}(s_3)\quad\text{for $\r\in R$}.
\end{equation}
We notice an important difference between 5-cyclidic harmonics of the first and second kind (considered in Section 3).
The external 5-cyclidic harmonics of the first kind are simply the Kelvin transformations of the internal 5-cyclidic harmonics of the first kind up to a constant factor.
There is no such simple relationship between internal and external 5-cyclidic harmonics of the second kind.

\begin{thm}\label{5:t2}
Let $\n\in\N_0^2$ and $\p=(p_1,p_2,p_3)\in\{0,1\}^3$. Then $\H{1}$
is harmonic on $\R^3\setminus K_1$.
The functions $\H{1}$ share the symmetries \eqref{5:sym5} with $\G{1}$.
Moreover,
\begin{equation}\label{5:est1}
\H{1}(\r) =O(\|\r\|^{-1})\quad\text{as $\|\r\|\to\infty$},
\end{equation}
and
\begin{equation}\label{5:est2}
\|\nabla \H{1}(\r)\| =O(\|\r\|^{-2})\quad\text{as $\|\r\|\to\infty$} .
\end{equation}
\end{thm}
\begin{proof}
The proof of analyticity and symmetry follows directly from \eqref{5:H} and Theorem \ref{5:t1}.
Estimates \eqref{5:est1} and \eqref{5:est2} follow from the fact that the Kelvin transformation of $\H{1}$ is analytic at the origin.
\end{proof}

\section{Expansion of the reciprocal distance in 5-cyclidic harmonics of first kind}\label{expansion1}

For fixed $s\in(a_0,a_1)$ the coordinate surface \eqref{2:surface} consists of two closed surfaces of genus $0$. One lies inside the unit ball $B_1(\0)$ and the other one is
obtained from it by inversion $\sigma_0$.
We consider the region $D_1$ interior to the coordinate surface $s=d_1$ which lies in $B_1(\0)$:
\begin{equation}\label{6:ring1}
D_1:=\{\r\in \R^3: \|\r\|<1, s_1>d_1\}.
\end{equation}

\begin{thm}\label{6:t1}
Let $d_1\in(a_0,a_1)$, $\n\in\N_0^2$, $\p\in\{0,1\}^3$. Then
\begin{equation}\label{6:integralformula}
\H{1}(\r')=\frac{1}{4\pi\omega(d_1)\{\E{1}{1}(d_1)\}^2}
\int_{\partial D_1} \frac{\G{1}(\r)}{h_1(\r)\|\r-\r'\|}\, dS(\r)
\end{equation}
for all $\r'\in\R^3\setminus\bar D_1$.
The scale factor $h_1$ is given by
\begin{equation}\label{6:h2}
 16 \{h_1(\r)\}^2=\frac{(\|\r\|^2-1)^2}{(d_1-a_0)^2}+ \frac{4x^2}{(d_1-a_1)^2}+\frac{4y^2}{(d_1-a_2)^2}+ \frac{4z^2}{(d_1-a_3)^2}.
\end{equation}
\end{thm}
\begin{proof}
The proof is similar to the proof of Theorem \ref{4:t1}. We use \eqref{5:G0}, \eqref{5:H1} and the Wronskian \eqref{5:Wronskian}.
\end{proof}

We obtain the expansion of the reciprocal distance in 5-cyclidic harmonics of first kind.

\begin{thm}\label{6:t2}
Let $\r,\r'\in\R^3$ with 5-cyclidic coordinates $s_1, s_1'$, respectively.
If either (a) $\|\r\|, \|\r'\|\le1$, $s_1>s_1'$, or (b) $\|\r\|<1<\|\r'\|$, or (c)   $\|\r\|,\|\r'\|\ge1$, $s_1<s_1'$, then
\begin{equation}\label{6:expansion}
\frac{1}{\|\r-\r'\|}=2\pi\sum_{\n\in\N_0^2} \sum_{\p\in\{0,1\}^3}  \G{1}(\r)\H{1}(\r').
\end{equation}
\end{thm}
\begin{proof}
Suppose (a) or (b) holds. Pick $d_1$ such that $s_1'<d_1<s_1$ if (a) holds, or such that $a_0<d_1<s_1$ if (b) holds. Then consider the domain $D_1$ defined in \eqref{6:ring1}.
The function $f(\q):=\|\q-\r'\|^{-1}$ is harmonic on an open set containing $\bar D_1$.
Therefore, by \cite[(71),(73)]{CohlVolkcyclide1}, we have
\begin{equation}\label{6:expansion2}
\frac{1}{\|\r-\r'\|}= \sum_{\n\in\N_0^2} \sum_{\p\in\{0,1\}^3} \d{1}\G{1}(\r),
\end{equation}
where
\[ \d{1}=\frac{1}{2\omega(d_1)\{\E{1}{1}(d_1)\}^2} \int_{\partial D_1} \frac{\G{1}(\q)}{h_1(\q)\|\q-\r'\|}\,dS(\q).\]
Using Theorem \ref{6:t1}, we obtain \eqref{6:expansion}.

Now suppose (c) holds. Then the points $\sigma_0(\r')$, $\sigma_0(\r)$ in place of $\r,\r'$ satisfy (a), so, by what we already proved,
\[ \frac{1}{\|\sigma_0(\r)-\sigma_0(\r')\|}=2\pi\sum_{\n\in\N_0^2} \sum_{\p\in\{0,1\}^3}  \G{1}(\sigma_0(\r'))\H{1}(\sigma_0(\r)). \]
This gives \eqref{6:expansion} by using \eqref{5:H} and observing that
\[ \|\r-\r'\|=\|\r\|\|\r'\|\|\sigma_0(\r)-\sigma_0(\r')\| .\]
\end{proof}

\section{5-cyclidic harmonics of the third kind}\label{harmonics3}

The 5-cyclidic harmonics of the third kind are treated analogously to the harmonics of the first kind.
Therefore, we will omit all proofs in the following two sections.

In \cite[Section IX]{CohlVolkcyclide1} we introduced special solutions $w_i(s_i)=\E{3}{i}(s_i)$ to equation \eqref{2:Fuchs} for eigenvalues $\lambda_j=\lambda^{(3)}_{j,\n,\p}$, $j=1,2$, for every $\n\in\N_0^2$, $\p=(p_0,p_1,p_2)\in\{0,1\}^3$.
These functions have the form
\begin{eqnarray*}
 \E{3}{i}(s_i)&=&(s_i-a_{i-1})^{p_{i-1}/2}(a_i-s_i)^{p_i/2} \EE{3}{i}(s_i),\quad s_i\in(a_{i-1},a_i),\ i=1,2,\\
 \E{3}{3}(s_3)&=&(s_3-a_2)^{p_2/2} \EE{3}{3}(s_3), \quad s_3\in(a_2,a_3),
\end{eqnarray*}
where $\EE{3}{i}$ is analytic on $[a_{i-1},a_i]$ for $i=1,2$ while $\EE{3}{3}$ is analytic on $[a_2,a_3)$.
As in \cite[Section X]{CohlVolkcyclide1} we define the internal 5-cyclidic harmonic of the third kind by
\begin{equation}\label{7:G0}
\G{3}(\r):=(\|\r\|^2+1)^{-1/2}  \E{3}{1}(s_1)\E{3}{2}(s_2)\E{3}{3}(s_3), \quad \r\in R.
\end{equation}

Let $\P{3}(s_3)$ be the solution to \eqref{2:Fuchs} (with $\lambda_j=\lambda^{(3)}_{j,\n,\p}$)
on $(a_2,a_3)$ belonging to the exponent $0$ at $s=a_3$ and uniquely determined by the condition $\P{3}(a_3)=1$.
We write
\[ \P{3}(s_3)=(s_3-a_2)^{p_2/2} \PP{3}(s_3),\quad s_3\in(a_2,a_3),\]
where $\PP{3}(s_3)$ is analytic on $(a_2,a_3]$.
Then we define
\begin{equation}\label{7:I}
\I{3}(\r):=(\|\r\|^2+1)^{-1/2}\prod_{j=0}^2 (\chi_j(\r))^{p_j} \prod_{i=1}^2 \EE{3}{i}(s_i) \PP{3}(s_3)\quad \text{if $s_3\ne a_2$} .
\end{equation}
The condition $s_3\ne a_2$ is equivalent to $\r\in \R^3\setminus(K_2\cup M_2)$; see Figure \ref{2:fig2}.

Similarly, let $\Q{3}(s_3)$ be the solution to \eqref{2:Fuchs} (with $\lambda_j=\lambda^{(3)}_{j,\n,\p}$)
on $(a_2,a_3)$ belonging to the exponent $\frac12$ at $s=a_3$ and uniquely determined by the condition $\lim_{s_3\to a_3^-} \omega(s_3)\frac{d}{ds_3} \Q{3}(s_3)=1$.
We write
\[ \Q{3}(s_3)=(a_3-s_3)^{1/2}(s_3-a_2)^{p_2/2}\QQ{3}(s_3),\quad s_3\in(a_2,a_3),\]
where $\QQ{3}(s_3)$ is analytic on $(a_2,a_3]$.
Then we define
\begin{equation}\label{7:J}
\J{3}(\r):=(\|\r\|^2+1)^{-1/2}\chi_3(\r)\prod_{j=0}^2 (\chi_j(\r))^{p_j} \prod_{i=1}^2 \EE{3}{i}(s_i)\QQ{3}(s_3) \quad \text{if $s_3\ne a_2$} .
\end{equation}

\begin{lemma}\label{7:l1}
The functions $\I{3}$ and $\J{3}$ are harmonic on $\R^3\setminus(K_2\cup M_2)$. They have the symmetries
\begin{eqnarray}
 \I{3}(\sigma_0(\r))&=&(-1)^{p_0}\|\r\|\I{3}(\r),\label{7:sym1}\\
 \I{3}(\sigma_j(\r))&=&(-1)^{p_j} \I{3}(\r),\quad j=1,2, \label{7:sym2}\\
 \I{3}(\sigma_3(\r))&=&\I{3}(\r),\label{7:sym3}\\
 \J{3}(\sigma_0(\r))&=&(-1)^{p_0}\|\r\|\J{3}(\r),\label{7:sym4}\\
 \J{3}(\sigma_j(\r))&=&(-1)^{p_j} \J{3}(\r) ,\quad j=1,2, \label{7:sym5}\\
 \J{3}(\sigma_3(\r))&=-&\J{3}(\r).\label{7:sym6}
\end{eqnarray}
\end{lemma}

Since $\P{3},\Q{3}$ form a fundamental system of solutions to \eqref{2:Fuchs} (with $\lambda_j=\lambda^{(3)}_{j,\n,\p}$) on $(a_2,a_3)$, there are (nonzero) scalars $\a{3}$,
$\b{3}$ such that
\[ \E{3}{3}=\a{3} \P{3}+\b{3} \Q{3} .\]
This leads to the global definition of internal 5-cyclidic harmonics of the third kind
\begin{equation}\label{7:G}
\G{3}:=\a{3}\I{3}+\b{3}\J{3} .
\end{equation}
If $z>0$, we can write $\G{3}$ as follows
\begin{equation}\label{7:G1}
 \G{3}(\r)=(\|\r\|^2+1)^{-1/2}\prod_{j=0}^2 (\chi_j(\r))^{p_j} \prod_{i=1}^3 \EE{3}{i}(s_i) .
\end{equation}

\begin{thm}\label{7:t1}
Let $\n\in\N_0^2$ and $\p=(p_0,p_1,p_2)\in\{0,1\}^3$. Then $\G{3}$ extends continuously to a harmonic function on $\R^3\setminus M_2$.
Moreover
\begin{eqnarray}
  \G{3}(\sigma_0(\r))&=&(-1)^{p_0}\|\r\|\G{3}(\r),\label{7:sym7}\\
  \G{3}(\sigma_j(\r))&=&(-1)^{p_j} \G{3}(\r),\quad j=1,2. \label{7:sym8}
\end{eqnarray}
\end{thm}

We introduce another solution of \eqref{2:Fuchs} by
\begin{equation}\label{7:F}
\F{3}(s_3)= \c{3}\left(\a{3} \P{3}(s_3)-\b{3} \Q{3}(s_3)\right),\quad s_3\in (a_2,a_3) .
\end{equation}
We determine $\c{3}$ from the  Wronskian
\begin{equation}\label{7:Wronskian}
\omega(s_3)\left(\E{3}{3}(s_3)\frac{d}{ds_3}\F{3}(s_3)-\F{3}(s_3)\frac{d}{ds_3}\E{3}{3}(s_3)\right) =1,
\end{equation}
which is equivalent to
\[ \c{3}=\frac{-1}{2\a{3}\b{3}} .\]
We define external 5-cyclidic harmonics of the third kind by
\begin{equation}\label{7:H}
\H{3}(\r):=\c{3}\G{3}(\sigma_3(\r)) \quad \text{for $\r\in\R^3\setminus K_2$}.
\end{equation}
In particular, we have
\begin{equation}\label{7:H1}
\H{3}(\r)=(\|\r\|^2+1)^{-1/2} \E{3}{1}(s_1)\E{3}{2}(s_2)\F{3}(s_3)\quad\text{for $\r\in R$.}
\end{equation}

\begin{thm}\label{7:t2}
Let $\n\in\N_0^2$ and $\p=(p_0,p_1,p_2)\in\{0,1\}^3$. Then $\H{3}$
is harmonic on $\R^3\setminus K_2$. The functions $\H{3}$ share the symmetries \eqref{7:sym7}, \eqref{7:sym8} with $\G{3}$.
Moreover,
\begin{equation}\label{7:est1}
\H{3}(\r) =O(\|\r\|^{-1})\quad\text{as $\|\r\|\to\infty$},
\end{equation}
and
\begin{equation}\label{7:est2}
\|\nabla \H{3}(\r)\| =O(\|\r\|^{-2})\quad\text{as $\|\r\|\to\infty$} .
\end{equation}
\end{thm}

\section{Expansion of the reciprocal distance in 5-cyclidic harmonics of third kind}\label{expansion3}

For fixed $s\in(a_2,a_3)$ the coordinate surface \eqref{2:surface} consists of two closed surfaces of genus $0$. One lies in the half-space $z>0$ and the other one is
obtained from it by reflection at the plane $z=0$.
We consider the region interior to the coordinate surface $s=d_3$ which lies in the half-space $\{\r: z>0\}$:
\begin{equation}\label{8:solid}
D_3:=\{\r\in \R^3: z>0, s_3<d_3\}.
\end{equation}

\begin{thm}\label{8:t1}
Let $d_3\in(a_2,a_3)$, $\n\in\N_0^2$, $\p\in\{0,1\}^3$. Then
\begin{equation}\label{8:integralformula}
\H{3}(\r')=\frac{1}{4\pi\omega(d_3)\{\E{3}{3}(d_3)\}^2}
\int_{\partial D_3} \frac{\G{3}(\r)}{h_3(\r)\|\r-\r'\|}\, dS(\r)
\end{equation}
for all $\r'\in\R^3\setminus\bar D_3$.
The scale factor $h_3$ is given by
\begin{equation}\label{8:h3}
 16 \{h_3(\r)\}^2=\frac{(\|\r\|^2-1)^2}{(d_3-a_0)^2}+ \frac{4x^2}{(d_3-a_1)^2}+\frac{4y^2}{(d_3-a_2)^2}+ \frac{4z^2}{(d_3-a_3)^2}.
\end{equation}
\end{thm}

We obtain the expansion of the reciprocal distance in 5-cyclidic harmonics of the third kind.

\begin{thm}\label{8:t2}
Let $\r=(x,y,z),\r'=(x',y',z')\in\R^3$ with 5-cyclidic coordinates $s_3, s_3'$, respectively.
If either (a) $z,z'\ge 0$, $s_3<s_3'$, or (b) $z'<0<z$, or (c)   $z,z'\le0$, $s_3'<s_3$, then
\begin{equation}\label{8:expansion}
\frac{1}{\|\r-\r'\|}=2\pi\sum_{\n\in\N_0^2} \sum_{\p\in\{0,1\}^3}  \G{3}(\r)\H{3}(\r').
\end{equation}
\end{thm}



\end{document}